\documentclass[10pt]{amsart}

\usepackage{amsmath, amssymb, euscript,  enumerate}
\usepackage{color}
\usepackage{epsfig,graphicx}
\usepackage{ulem}
\usepackage{cite}
\usepackage{pxfonts}
\usepackage{color,wrapfig}

\usepackage{hyperref}
\usepackage{url}

  \newtheorem{thm}{Theorem}
  \newtheorem{thm2}{Theorem}
 \newtheorem{thm3}{Theorem}
 \newtheorem{thm4}{Theorem}
   \newtheorem{example}[thm2]{Example}
   \newtheorem{remark}[thm3]{Remark}
  \newtheorem{lemma}[thm4]{Lemma}

\begin{document}

\title[Linear Superposition of Minimal Surfaces: Generalized Helicoids and Minimal Cones]{Linear Superposition of Minimal Surfaces: \\ Generalized Helicoids and Minimal Cones} 
  
\bigskip

  \author{Jens Hoppe} 
  \address{KTH Royal Institute of Technology, SE-100 44, Stockholm, Sweden}
\email{hoppe@kth.se}


\begin{abstract}
Observing a superposition principle for independent $\infty$-harmonic functions, a family 
of new minimal hypersurfaces in Euclidean space is found, and that 
(linear combinations of) higher dimensional helicoids 
induce new algebraic minimal cones of arbitrarily high degree, 
generalizing Lawson's minimal cubic cone 
$y_{2} \left(  {x_{1}}^2 - {y_{1}}^2 \right)  -  x_{2} \left( 2 x_1 y_1 \right) =0 $ in ${\mathbb{R}}^{4}$,
and also including a family of Tkachev's minimal cubic cones.
\end{abstract}


\maketitle
 
Examples of  algebraic minimal cones include those
constructed by  E. Cartan \cite{Cartan1939}, W. Hsiang \cite{H1967}, H. Lawson \cite{Lawson1970}, D. Ferus, H. Karcher, H. F. M\"{u}nzner \cite{FKM1981},
V. Tkachev \cite{T2010, T2010b}, and G. Linardopoulos, T. Turgut, et al \cite{HLT2016}. 
 
In 2013, generalizing the classical helicoid $z=\arctan\frac{y}{x}$ which is foliated by straight lines in 
${\mathbb{R}}^{3}$, 
new minimal hypersurfaces in ${\mathbb{R}}^{2n+1}$ were constructed by J.Choe and the author \cite{CH2013}  that can be thought of as arising by a multi-screw motion from the Clifford cone in ${\mathbb{R}}^{2n}$

In this article, apart from pointing out an important effective linear superposition principle for a subclass of minimal surfaces,
a multitude of new algebraic minimal cones is constructed,
induced from linear combinations of generalized/higher dimensional helicoids. 

\begin{example}[\textbf{CH-helicoid in ${\mathbb{R}}^{2n+1}$, \cite[Theorem 2]{CH2013}}]    \label{E1}
Let $\lambda \in \mathbb{R}$. Sweeping out the Clifford cone ${\mathbf{C}}^{2n-1}$ in  ${\mathbb{R}}^{2n}$ given by 
\[
   {\mathbf{C}}^{2n-1} = \left\{ \; 
\begin{bmatrix}
  \, p_{1}   \, \\ 
  \, q_{1}   \,   \\ 
\vdots \\
  \, p_{n}    \,  \\ 
  \, q_{n}  \,   \\ 
\end{bmatrix}  \in   {\mathbb{R}}^{2n}  \; \; \vert \; \;   {p_{1}}^{2} + \cdots +  {p_{n}}^{2}  = {q_{1}}^{2} + \cdots +  {q_{n}}^{2}  \;
\right\}
\]
yields the CH-helicoid ${\mathcal{H}}_{\lambda}$  in  ${\mathbb{R}}^{2n+1}$ given by 
\[
{\mathcal{H}}_{\lambda}  = \left\{ \; 
\begin{bmatrix}
   x_{1} \\ 
  y_{1}  \\ 
\vdots \\
 x_{n} \\ 
 y_{n}   \\ 
z \\
\end{bmatrix} =
\begin{bmatrix}
  \, p_{1} \cos \Theta  - q_{1} \sin \Theta  \, \\ 
  \, q_{1} \cos \Theta  +p _{1} \sin \Theta  \,   \\ 
\vdots \\
  \, p_{n} \cos \Theta  - q_{n} \sin \Theta   \,  \\ 
  \, q_{n} \cos \Theta  +p _{n} \sin \Theta   \,    \\ 
  \lambda  \, \Theta \\
\end{bmatrix}  \in   {\mathbb{R}}^{2n+1}  \; \vert \;    \Theta \in \mathbb{R}, \; \begin{bmatrix}
  \, p_{1}   \, \\ 
  \, q_{1}   \,   \\ 
\vdots \\
  \, p_{n}    \,  \\ 
  \, q_{n}  \,   \\ 
\end{bmatrix}  \in {\mathbf{C}}^{2n-1} \;
\right\}.
\]
One can check that ${\mathcal{H}}_{\lambda \neq 0}$ is congruent to  
$\lambda {\mathcal{H}}_{1}$. As observed in \cite[Remark 5]{LL2014}, one finds that the minimal variety ${\mathcal{H}}_{1}$ can be viewed as 
 the multi-valued graph in ${\mathbb{R}}^{2n+1}={\mathbb{C}}^{n} \times \mathbb{R}$ 
\[
 z= \mathbf{arg} \, \left(  \sqrt{  \; (x_{1}+i y_{1})^{2}  + \cdots +  (x_{n}+i y_{n})^{2}  \;  } \;  \right).
\]
\end{example}

\begin{lemma}[\textbf{$\infty$-harmonicity of the height function of the CH-helicoid}] \label{harmonicity}
The CH-function in ${\mathbb{R}}^{2n}={\mathbb{R}}^{n} \times {\mathbb{R}}^{n}$ given by  
\[
 F ={\mathbf{F}}_{{\mathbb{R}}^{2n}}(X, Y) = \frac{1}{2} \arctan  \left( \frac{ V  }{ U } \right), \quad \text{where} \quad
 (U, V) = \left( \, {\Vert X \Vert}_{{}_{{\mathbb{R}}^{n}}}^{\;2} 
 - {\Vert Y \Vert}_{{}_{{\mathbb{R}}^{n}}}^{\;2} , \, 2 {\langle X, Y \rangle}_{{}_{{\mathbb{R}}^{n}}} \, \right),
\]
is $\infty$-harmonic, in the sense that it solves the so called infinity Laplace equation 
\[
0 = {\Delta}_{\infty} F :=  {\big\langle \; \nabla F, \, \nabla \left( \frac{1}{2} { {\Vert  \nabla F   \Vert}_{{}_{{\mathbb{R}}^{2n}}} }^{2} \right) \; \big\rangle}_{
{}_{  {\mathbb{R}}^{2n} }}.
\]
\end{lemma}

\begin{proof}
Straightforward computation. 
\end{proof}

We present an alternative proof of the minimality of the CH-helicoid, and shall see that the analytic proof motivates Theorem \ref{main 1}.

\begin{lemma}[\textbf{Minimality of the CH-helicoid}]   \label{chminimal}
The multi-valued graph 
\[ 
z= \mathbf{arg} \, \left(  \sqrt{  \; (x_{1}+i y_{1})^{2}  + \cdots +  (x_{n}+i y_{n})^{2}  \;  } \;  \right)
\]
has zero mean curvature in Euclidean space ${\mathbb{R}}^{2n+1}={\mathbb{C}}^{n} \times \mathbb{R}$.
\end{lemma}

\begin{proof} One needs to check that the induced height function $F$ solves 
 the minimal hypersurface equation
\[
 0=  {\nabla} \cdot \left(   \frac{\nabla F }{ \sqrt{1+{ \Vert \nabla F \Vert}^2 \,} \, } \right)=  \frac{  \left(  1+{ \Vert \nabla F \Vert}^2\right)  \Delta F 
 -  {\Delta}_{\infty} F }{ { \left( 1+{ \Vert \nabla F \Vert}^2 \right)}^{ \frac{3}{2}} }.
\]
 It is straightforward to check $\Delta F=0$ ( and ${\Delta}_{\infty} F=0$, by Lemma 1) 
\end{proof}

\begin{remark}
In \cite[Theorem 2]{Ar1968}, G. Aronsson showed that if a two-variable harmonic function $F(x,y)$ is also $\infty$-harmonic, then $F(x,y)=ax+by+c$ or $F(x,y)=d\arctan\left( \frac{y-y_{0}}{x-x_{0}}\right) +e$ for some real constants $a$, $b$, $c$, $d$, $e$, $x_{0}$, $y_{0}$. See also the classification result due to W. C. Graustein \cite{Gr1940}. Both references include interesting hydrodynamic interpretations.
\end{remark}

 \begin{remark}
The CH-function is $p$-harmonic as it solves the $p$-Laplace equation 
\[
0={\Delta}_{p} F :=\nabla \cdot \left( \, {\Vert \nabla F \Vert}^{\, p-2} \nabla F \, \right) = {\Vert \nabla F \Vert}^{\, p-4} \left[ \, (p-2) \,
{\Delta}_{\infty} F +  {\Vert \nabla F \Vert}^{2} \, \Delta F  \, \right].
\] 
\end{remark}

\begin{lemma}[\textbf{Superposition principle for $\infty$-Laplacian operator}] \label{superposition}
Let ${\mathbf{F}}_{A} \left(u_{1}, \cdots, u_{A} \right)$ and ${\mathbf{F}}_{B} \left(v_{1}, \cdots, v_{B} \right)$ be two independent ${\mathcal{C}}^{2}$ 
real valued functions. For real constants $\mu_{A}$ and $\mu_{B}$, we associate the combination
 ${\mathbf{F}}  \left(u_{1}, \cdots, u_{A}, v_{1}, \cdots, v_{B} \right) :=\, {\mu_{A}}^{\frac{1}{3}}  \, {\mathbf{F}}_{A}  + 
   \, {\mu_{B}}^{\frac{1}{3}}  \, {\mathbf{F}}_{B} $. Then, 
\[
{\Delta}_{\infty, {\mathbb{R}}^{A+B}}   {\mathbf{F}}  
=  {\mu}_{A} \, {\Delta}_{\infty, {\mathbb{R}}^{A}} {\mathbf{F}}_{A} +  {\mu}_{B} \, {\Delta}_{\infty, {\mathbb{R}}^{B}} {\mathbf{F}}_{B}.
\]
\end{lemma}

\begin{thm}[\textbf{Generalized helicoids in odd dimensional  Euclidean space}]   \label{main 1}
The multi-valued graph of an arbitrary finite number of linear combination of independent CH height functions  (introduced in Lemma \ref{harmonicity}) becomes a minimal hypersurface.
\end{thm}

\begin{proof} From Lemma \ref{harmonicity}, \ref{chminimal}, and \ref{superposition}, one sees that a linear combination of independent CH 
functions should be both harmonic and $\infty$-harmonic. 
\end{proof}

\begin{example}[\textbf{Superposition principle for generalized helicoids}]    \label{E2}
Let $\mu_{A}$, $\mu_{B}$, $\mu_{C} \in \mathbb{R}$ be constants. By Theorem \ref{main 1}, the multi-valued graph
 \[
 z= \mu_{A} \, \mathbf{arg}  \left(  \sqrt{   (x_{1}+i y_{1})^{2}  +\cdots + (x_{4}+i y_{4})^{2} }\, \right) 
 + \mu_{B} \,  \mathbf{arg} ( x_{5} +iy_{5}) + {\mu}_{C} \,  \mathbf{arg}  ( x_{6} +iy_{6} ).
\]
 is minimal  in ${\mathbb{R}}^{13}={\mathbb{C}}^{4} \times {\mathbb{C}}  \times {\mathbb{C}}  \times {\mathbb{R}}$.
 It is invariant under the multi-screw motion
 \[
   \left( {\zeta}_{1}, \cdots , {\zeta}_{4}, {\zeta}_{5}, {\zeta}_{6}, z \right) \rightarrow   \left(\, e^{i \,t_{A}} {\zeta}_{1}, \cdots, e^{i \,t_{A}} {\zeta}_{4},  \, e^{i \, t_{B}}{\zeta}_{5}, \, e^{i \, t_{C}}{\zeta}_{6}, z+ {\mu}_{A}  t_{A}+ {\mu}_{B} t_{B} + {\mu}_{C} t_{C} \,\right),
    \]
where ${\zeta}_{k}=x_{k} +i y_{k}$ for $k \in \left\{1, \cdots, 6\right\}$. 
\end{example} 


\begin{thm}[\textbf{$n$-variable harmonic and $\infty$-harmonic functions induce minimal hypersurfaces in ${\mathbb{R}}^{n+1}$ and ${\mathbb{R}}^{n+2}$}]   \label{main 2}
Let $\mathbf{f} \left(z_{1}, \cdots, z_{n} \right)$ is a ${\mathcal{C}}^{2}$-function satisyfing 
\[
\Delta \mathbf{f}=0  \;\; \text{and} \;\;  {\Delta}_{\infty} \mathbf{f}=0.
\]
\begin{enumerate}
\item The graph $z_{n+1}=\mathbf{f} \left(z_{1}, \cdots, z_{n} \right)$ is a minimal hypersurface in  ${\mathbb{R}}^{n+1}$.
\item The graph $z_{n+1}= z_{0} \; \tan \mathbf{f} \left(z_{1}, \cdots, z_{n} \right)$ is a minimal hypersurface in  ${\mathbb{R}}^{n+2}$. 
\end{enumerate}
\end{thm}

\begin{proof}
The proofs for the two items use the same idea.
First verify the item (1) using the superposition principle. The idea is to view the graph $z_{n+1}=\mathbf{f} \left(z_{1}, \cdots, z_{n} \right)$
as a level set of an $(n+1)$-variable function $\mathbf{U} \left(z_{1}, \cdots, z_{n}, z_{n+1} \right) := -z_{n+1} + \mathbf{f} \left(z_{1}, \cdots, z_{n} \right)$: 
\[
\left\{ \;  \left(z_{1}, \cdots, z_{n}, z_{n+1} \right)  \in {\mathbb{R}}^{n+1}  \;  \vert \;   0= \mathbf{U} \left(z_{1}, \cdots, z_{n}, z_{n+1} \right) \; \right\}.
\]
This level set is a minimal submanifold if and only if the function $\mathbf{U}$ satisfies 
\[
0 = \sum_{k=1}^{n+1}  \frac{\partial }{\partial z_{k}}  \left(    \frac{    \frac{\partial \mathbf{U}}{\partial z_{k}}    }{  \, {\Vert \nabla \mathbf{U} \Vert}  \, }         \right), 
 \quad \text{or equivalently,} \quad
 0=  {\Vert \nabla \mathbf{U} \Vert}^2 \Delta \mathbf{U} - \Delta_{\infty} \mathbf{U}.
\]
Since both $-z_{n+1}$ and $\mathbf{f} \left(z_{1}, \cdots, z_{n} \right)$ are harmonic and $\infty$-harmonic, by the superposition principle, its sum $\mathbf{U}$
is harmonic and $\infty$-harmonic. This guarantees the desired equality $0=  {\Vert \nabla \mathbf{U} \Vert}^2 \Delta \mathbf{U} - \Delta_{\infty} \mathbf{U}$. 
For the proof of item (2), introduce the splitting 
\[
\mathbf{V} \left(z_{0}, z_{1}, \cdots, z_{n}, z_{n+1} \right) := - arctan \left( \frac{z_{n+1}}{z_{0}} \right)+ \mathbf{f} \left(z_{1}, \cdots, z_{n} \right).
\]
Since both $- arctan \left( \frac{z_{n+1}}{z_{0}} \right)$ and  $\mathbf{f} \left(z_{1}, \cdots, z_{n} \right)$ are harmonic and $\infty$-harmonic, by the superposition principle, 
its sum $\mathbf{V}$ is also  harmonic and $\infty$-harmonic. Thus, the zero set $\left\{ \;  \left(z_{0}, z_{1}, \cdots, z_{n}, z_{n+1} \right)  \in {\mathbb{R}}^{n+2}  \;  \vert \;  
 0= \mathbf{V} \left(z_{0}, z_{1}, \cdots, z_{n}, z_{n+1} \right) \; \right\}$ is minimal in  ${\mathbb{R}}^{n+2}$.
 \end{proof}

\begin{remark} The same argument of the proof of Theorem \ref{main 2} shows that whenever two  ${\mathcal{C}}^{2}$-functions $\mathbf{f} \left(z_{1}, \cdots, z_{n} \right)$ and
$\mathbf{g} \left(w_{1}, \cdots, w_{m} \right)$ are  harmonic and $\infty$-harmonic, the level set 
\[
\left\{ \;  \left(z_{1}, \cdots, z_{n}, w_{1}, \cdots, w_{m} \right)  \in {\mathbb{R}}^{n+m}  \;  \vert \;   0=\mathbf{f} \left(z_{1}, \cdots, z_{n} \right) + \mathbf{g} \left(w_{1}, \cdots, w_{m} \right) \; \right\}.
\]
defines a (possibly singular) minimal submanifold in Euclidean space ${\mathbb{R}}^{n+m}$.
\end{remark}

\begin{example}[\textbf{Helicoids in ${\mathbb{R}}^{3}$ and Clifford's quadratic cone in ${\mathbb{R}}^{4}$}]  \label{E3} The height function $\mathbf{f} \left( x_1, y_1 \right)= arctan \left( 
\frac{y_1}{x_1} \right)$ of the helicoid $z=arctan \left(\frac{y_1}{x_1} \right)$ in ${\mathbb{R}}^{3}$ is both harmonic and $\infty$-harmonic. Then, item (2) of Theorem \ref{main 2} guarantees that 
the graph $\Sigma$ given by $y_{2}= x_{2} \tan  \mathbf{f} \left( x_1, y_1 \right)$ becomes a minimal hypersurface in ${\mathbb{R}}^{4}$. The hypersurface $\Sigma$ is the quadratic
 cone $y_2 x_1 = x_2 y_1$. In fact, applying the $\frac{\pi}{4}$-rotations in $x_1 y_2$-plane and $x_{1}y_{2}$-plane, one finds that $\Sigma$ is congruent to the 
hypersurface ${y_{2}}^{2} - {x_{1}}^2 = {y_{1}}^{2} - {x_{2}}^2$. The $3$-fold  $\Sigma$ is the cone over Clifford's minimal torus ${x_{1}}^{2}+{y_{1}}^{2} ={x_{2}}^{2}+{y_{2}}^{2} =\frac{1}{2} $ in  
the sphere ${\mathbb{S}}^{3} \subset {\mathbb{R}}^{4}$. 
\end{example} 

\begin{example}[\textbf{Lawson's algebraic minimal cones \cite{Lawson1970} of degree $N+1$ in ${\mathbb{R}}^{4}$}]  \label{E4} Let $N \geq 1$ be an integer. 
Apply the dilation $\left(x_{1}, y_{1}, z_{1} \right)$
$\to$ $\left( \frac{x_{1}}{N}, \frac{y_{1}}{N}, \frac{z_{1}}{N} \right)$ to the minimal surface $z=arctan \left(\frac{y_1}{x_1} \right)$ in ${\mathbb{R}}^{3}$ to obtain the helicoid  $z={\mathbf{f}}_{N} \left( x_1, y_1 \right) =N \, arctan \left(\frac{y_1}{x_1} \right)$. Since the height function ${\mathbf{f}}_{N} \left( x_1, y_1 \right)= N \,arctan \left( \frac{y_1}{x_1} \right)$   is both harmonic and $\infty$-harmonic,  
introducing the rational polynomial ${\mathbf{q}}_{N}(t) = \tan \left( N \, arctan \; t \right) \in {\mathbb{Q}}[t]$, the item (2) of Theorem \ref{main 2} guarantees that 
the graph ${\Sigma}_{N}$ given by $y_{2}= x_{2} \mathbf{q} \left( \frac{y_{1}}{x_{1}} \right)$ is minimal  in ${\mathbb{R}}^{4}$. It is easy to check that 
${\Sigma}_{N}$ becomes a Lawson's algebraic minimal cone of degree $N+1$ in ${\mathbb{R}}^{4}$. As noticed in \cite[Theorem 3 and Proposition 7.2]{Lawson1970}, 
 the $3$-fold ${\Sigma}_{N}$ becomes the cone over a ruled minimal surface in the sphere ${\mathbb{S}}^{3} \subset {\mathbb{R}}^{4}$. 
In the case when $N=2$, one has  ${\mathbf{q}}_{2}(t) = \tan \left( 2 \, arctan \; t \right) = \frac{2t}{1 - t^2}$. Then one finds that the minimal $3$-fold ${\Sigma}_{2}$ is the cubic cone 
$y_{2} \left(  {x_{1}}^2 - {y_{1}}^2 \right)  -  x_{2} \left( 2 x_1 y_1 \right) =0 $ in ${\mathbb{R}}^{4}$.
 \end{example} 

\begin{example}[\textbf{CH-helicoids in ${\mathbb{R}}^{2n+1}$ and Tkachev's cubic cone in ${\mathbb{R}}^{2n+2}$}]    \label{E5}
The harmonicity and $\infty$-harmonicity of the height function 
\[
\mathbf{f} \left( x_{1}, y_{1}, \cdots, x_{n}, y_{n} \right) = arctan \left( \, \frac{  2   x_{1} y_{1}   + \cdots + 2 x_{n} y_{n}   }{    \left( {x_{1}}^{2} - {y_{1}}^{2}  \right) + \cdots +  \left( {x_{n}}^{2} - {y_{n}}^{2}  \right)   } \, \right)
\]
and the item (2) of Theorem \ref{main 2} guarantee that the cone
\[
y_{n+1} = x_{n+1} \left( \, \frac{ \, 2   x_{1} y_{1}   + \cdots + 2 x_{n} y_{n}  \,  }{    \left( {x_{1}}^{2} - {y_{1}}^{2}  \right) + \cdots +  \left( {x_{n}}^{2} - {y_{n}}^{2}  \right)   } \, \right)
\]
is minimal in ${\mathbb{R}}^{2n+2}$. It coincides with one of Tkachev's minimal cubic cones \cite[Example 4]{T2010}
(V. Tkachev used Clifford algebras to construct a large class of minimal cubic cones; see also \cite[Chapter 6]{NTV2014} and  \cite{T2010b}). 
.
\end{example}

\begin{example}[\textbf{Higher degree generalization of Tkachev's cubic cone in ${\mathbb{R}}^{2n+2}$}]     \label{E6}
For an integer $N \geq 1$, introduce the rational polynomial ${\mathbf{q}}_{N} (t) = \tan \left( N \, arctan \; t \right) \in {\mathbb{Q}}[t]$. Set
\begin{enumerate}
\item ${\mathbf{T}}_{1} \left( x_{1}, y_{1}, \cdots, x_{n}, y_{n} \right)=   \left( {x_{1}}^{2} - {y_{1}}^{2}  \right) + \cdots +  \left( {x_{n}}^{2} - {y_{n}}^{2}  \right)$, 
\item ${\mathbf{T}}_{2} \left( x_{1}, y_{1}, \cdots, x_{n}, y_{n} \right)= 2   x_{1} y_{1}   + \cdots + 2 x_{n} y_{n}$.
\end{enumerate}
Since the function ${\mathbf{f}}_{N} \left( x_{1}, y_{1}, \cdots, x_{n}, y_{n} \right) = N \, arctan \left( \, \frac{ \, {\mathbf{T}}_{2} \left( x_{1}, y_{1}, \cdots, x_{n}, y_{n} \right)  \,  }{ {\mathbf{T}}_{1} \left( x_{1}, y_{1}, \cdots, x_{n}, y_{n} \right)  } \, \right)$
is both harmonic and $\infty$-harmonic, the item (2) of Theorem 2 
guarantees that the algebraic cone
\[
y_{n+1} = x_{n+1} \, {\mathbf{q}}_{N} \left( \, \frac{ \,{\mathbf{T}}_{2} \left( x_{1}, y_{1}, \cdots, x_{n}, y_{n} \right) \,  }{  {\mathbf{T}}_{1} \left( x_{1}, y_{1}, \cdots, x_{n}, y_{n} \right)  } \, \right)
\]
is a minimal hypersurface in ${\mathbb{R}}^{2n+2}$. Consider the case  $N=2$. We find that  ${\mathbf{q}}_{2}(t) = \tan \left( 2 \, arctan \; t \right) = \frac{2t}{1 - t^2}$, and that 
the $(2n+1)$-fold ${\Sigma}_{2}$ is the minimal quintic cone in ${\mathbb{R}}^{2n+2}$:  
\[
y_{n+1} \left[    {(   {\mathbf{T}}_{1}  )}^2    -  {(  {\mathbf{T}}_{2}  )}^2   \right]   -  x_{n+1}  \left( 2  {\mathbf{T}}_{1} {\mathbf{T}}_{2} \right) =0.
\]
\end{example}

\begin{example}[\textbf{Minimal quintic cones in ${\mathbb{R}}^{4n+2}$}]    \label{E7}
Introduce quadratic polynomials
\begin{enumerate}
\item $ {\mathbf{T}}_{1} \left( x_{1}, y_{1}, \cdots, x_{n}, y_{n} \right)=   \left( {x_{1}}^{2} - {y_{1}}^{2}  \right) + \cdots +  \left( {x_{n}}^{2} - {y_{n}}^{2}  \right)$,  
\item ${\mathbf{T}}_{2} \left( x_{1}, y_{1}, \cdots, x_{n}, y_{n} \right)= 2   x_{1} y_{1}   + \cdots + 2 x_{n} y_{n}$, 
\item $ {\mathbf{S}}_{1} \left( x_{n+1}, y_{n+1}, \cdots, x_{2n}, y_{2n} \right)=   \left( {x_{n+1}}^{2} - {y_{n+1}}^{2}  \right) + \cdots +  \left( {x_{2n}}^{2} - {y_{2n}}^{2}  \right)$,  
\item ${\mathbf{S}}_{2} \left( x_{n+1}, y_{n+1}, \cdots, x_{2n}, y_{2n} \right)= 2   x_{n+1} y_{n+1}   + \cdots + 2 x_{2n} y_{2n}$. 
\end{enumerate}
From the proof of Theorem \ref{main 2} and the item (2) of Theorem \ref{main 2}, one knows that  the function 
\[
{\mathbf{f}} \left( x_{1}, y_{1}, \cdots, x_{2n}, y_{2n} \right)= \theta_{1}  + \theta_{2}, \quad \text{where} \;
 {\theta}_{1} = arctan \left( \frac{ {\mathbf{T}}_{2} }{ {\mathbf{T}}_{1} } \right), \;  \text{and} \; {\theta}_{2} = arctan \left( \frac{ {\mathbf{S}}_{2} }{ {\mathbf{S}}_{1} } \right)
\]
is both harmonic and $\infty$-harmonic, and obtain the minimal quintic  cone in ${\mathbb{R}}^{4n+2}$:
\[
   y_{2n+1} \left[  {\mathbf{T}}_1 {\mathbf{S}}_1 - {\mathbf{T}}_2 {\mathbf{S}}_2  \right] - x_{2n+1} \left[  {\mathbf{T}}_2 {\mathbf{S}}_1 + {\mathbf{T}}_1 {\mathbf{S}}_2 \right]  =0. 
\] 
\end{example}

\begin{example}[\textbf{Lawson's algebraic minimal cones in ${\mathbb{R}}^{2n}$}]     \label{E8}
Let $k_{1}$, $\cdots$, $k_{n}$ be positive integers with $gcd \left( k_{1}, \cdots, k_{n} \right)=1$. 
Observing that angle functions $arctan \left( \frac{y_{1}}{x_{1}} \right)$, $\cdots$,  $arctan \left( \frac{y_{n}}{x_{n}} \right)$ 
are harmonic and $\infty$-harmonic, one can associate the minimal hypersurface  in  ${\mathbb{R}}^{2n}$:
\[
 \left\{ \;  \left(x_{1}, y_{1}, \cdots, x_{n}, y_{n} \right)  \in {\mathbb{R}}^{2n}  \;  \vert \;  
 0=  k_{1} \, arctan \left( \frac{y_{1}}{x_{1}} \right) + \cdots  +  k_{n} \, arctan \left( \frac{y_{n}}{x_{n}} \right)  \; \right\}. 
\]
This level set produces the algebraic minimal cone in ${\mathbb{R}}^{2n}$: 
\[
  0 = \mathbf{Im} \, \left[  \, {\left( x_1 + i y_{1} \right)}^{k_{1}} \cdots {\left( x_n + i y_{n} \right)}^{k_{n}} \, \right].
\]
It is congruent to the example constructed by H. Lawson \cite[p. 352]{Lawson1970} 
after suitable rotations.
For instance, taking $n=3$ and $k_{1}=k_{2}=k_{3}=1$, one obtains the following minimal cubic cone in  ${\mathbb{R}}^{6}$:
\[
0 = y_1 x_2 x_3 + x_1 y_2 x_3 + x_1 x_2 y_3 - y_1 y_2 y_3.
\] 
\end{example}

\textbf{Acknowledgement}: 
The author is grateful to Hojoo Lee for collaboration, and to Vladimir Tkachev for sending him,
after completion of the above article, on June 27 his unpublished notes with V.Sergienko and (on June 28/29, before appearance) his extended June 28 arXive submission.

\bigskip



\begin{thebibliography}{00}

\bibitem{Ar1968}  G. Aronsson, \textit{On the partial differential equation ${u_x}^{2} u_{xx} + 2 u_x u_y u_{xy} +{u_y}^2 u_{yy}=0$}, 
Ark. Mat. \textbf{7} (1968) 395--425.
 
\bibitem{Cartan1939}
 E. Cartan, \textit{Sur des familles remarquables d'hypersurfaces isoparam\'{e}triques dans les espaces sph\'{e}riques}, 
Math. Z. \textbf{45} (1939). 335--367. 

\bibitem{CH2013} J. Choe, J. Hoppe,  \textit{Higher dimensional minimal submanifolds generalizing the catenoid and helicoid}, 
Tohoku Math. J. (2) \textbf{65} (2013), no. 1, 43--55.

 \bibitem{FKM1981}
 D. Ferus, H. Karcher, H. F. M\"{u}nzner, \textit{Cliffordalgebren und neue isoparametrische Hyperachen}, Math. Z (4) \textbf{177} (1981), 479--502.
 
\bibitem{Gr1940} W. C. Graustein, \textit{Harmonic minimal surfaces}, Trans. Amer. Math. Soc. \textbf{47} (1940), 173--206. 

\bibitem{HLT2016}
J. Hoppe, G. Linardopoulos, O. Teoman Turgut, 
\textit{New Minimal Hypersurfaces in ${\mathbb{R}}^{(k+ 1)(2k+ 1)}$ and ${\mathbb{S}}^{(2k+ 3) k}$}, arXiv preprint, arXiv:1602.09101 (2016).
  
\bibitem{H1967}
W. Hsiang, \textit{Remarks on closed minimal submanifolds in the standard Riemannian $m$-sphere},  J. Differential Geometry \textbf{1} (1967) 257--267. 

\bibitem{Lawson1970} H. B. Lawson, \textit{Complete minimal surfaces in ${\mathbb{S}}^{3}$},  Ann. of Math. (2) \textbf{92} (1970), 335--374.

\bibitem{LL2014} E. Lee, H. Lee, \textit{Generalizations of the Choe-Hoppe helicoid and Clifford cones 
in Euclidean space}, arXiv:1410.3418 (2014), to appear in J. Geom. Anal.

\bibitem{NTV2014} 
N. Nadirashvili, V. Tkachev, S. Vl\v{a}du\c{t}, Nonlinear elliptic equations and nonassociative algebras. Mathematical Surveys and Monographs, 200. American Mathematical Society, Providence, RI, 2014. viii+240 pp. ISBN: 978-1-4704-1710-9

\bibitem{T2010}
V. Tkachev, \textit{Minimal cubic cones via Clifford algebras}, Complex Anal. Oper. Theory \textbf{4} (2010), no. 3, 685--700.

\bibitem{T2010b}
V. Tkachev, \textit{On a classification of minimal cubic cones in ${\mathbb{R}}^{n}$},  arXiv preprint, arXiv:1009.5409 (2010).

\end{thebibliography}
\end{document}